\documentclass[12pt,oneside]{amsart}

\usepackage{graphicx}
\usepackage{amsfonts}
\usepackage{epsf}
\usepackage{amssymb}
\usepackage{amsmath}
\usepackage{amscd}
\usepackage{tikz}
\usepackage{pdfpages}
\usepackage{fancyhdr}
\usepackage{setspace}
\usepackage{hyperref}
\usepackage[all]{xy}
\usetikzlibrary{matrix}
\usepackage{verbatim}
\usepackage{enumerate}

\theoremstyle{definition}
\newtheorem{theorem}{Theorem}[section]
\newtheorem{proposition}[theorem]{Proposition}
\newtheorem{lemma}[theorem]{Lemma}
\newtheorem{definition}[theorem]{Definition}
\newtheorem{question}[theorem]{Question}
\newtheorem{corollary}[theorem]{Corollary}

\newtheorem{remark}[theorem]{Remark}

\newtheorem{claim}[theorem]{Claim}
\newtheorem{example}[theorem]{Example}

\newcommand{\Z}{\mathbb{Z}}

\newcommand{\N}{\mathbb{N}}

\newcommand{\R}{\mathbb{R}}

\newcommand{\id}{\text{id}}

\makeatletter
\newtheorem*{rep@theorem}{\rep@title}
\newcommand{\newreptheorem}[2]{%
\newenvironment{rep#1}[1]{%
 \def\rep@title{#2 \ref{##1}}%
 \begin{rep@theorem}}%
 {\end{rep@theorem}}}
\makeatother

\newreptheorem{theorem}{Theorem}
\newreptheorem{lemma}{Lemma}
\newreptheorem{question}{Question}
\newreptheorem{corollary}{Corollary}

\begin{document}

\rhead{\thepage}
\lhead{\author}
\thispagestyle{empty}


\raggedbottom
\pagenumbering{arabic}
\setcounter{section}{0}


\title{Fibered ribbon disks}
\author{Kyle Larson \and Jeffrey Meier}

\begin{abstract}
We study the relationship between fibered ribbon 1--knots and fibered ribbon 2--knots by studying fibered slice disks with handlebody fibers.  We give a characterization of fibered homotopy-ribbon disks and give analogues of the Stallings twist for fibered disks and 2--knots. As an application, we produce infinite families of distinct homotopy-ribbon disks with homotopy equivalent exteriors, with potential relevance to the Slice-Ribbon Conjecture.
We show that any fibered ribbon 2--knot can be obtained by doubling infinitely many different slice disks (sometimes in different contractible 4--manifolds).  Finally, we illustrate these ideas for the examples arising from spinning fibered 1--knots.

\end{abstract}
\maketitle
\section{Introduction}

A knot $K$ is fibered if its complement fibers over the circle (with the fibration well-behaved near $K$). Fibered knots have a long and rich history of study (for both classical knots and higher-dimensional knots). In the classical case, a theorem of Stallings (\cite{stallings:group}, see also \cite{neuwirth}) states that a knot is fibered if and only if its group has a finitely generated commutator subgroup. Stallings \cite{stallings:constructions} also gave a method to produce new fibered knots from old ones by twisting along a fiber, and Harer \cite{harer} showed that this twisting operation and a type of plumbing is sufficient to generate all fibered knots in $S^3$.

Another special class of knots are slice knots. A knot $K$ in $S^3$ is slice if it bounds a smoothly embedded disk in $B^4$ (and more generally an $n$--knot in $S^{n+2}$ is slice if it bounds a disk in $B^{n+3}$). If $K$ bounds an immersed disk in $S^3$ with only ribbon singularities we say $K$ is a ribbon knot. Every ribbon knot is slice, and the famous Slice-Ribbon Conjecture states that every slice knot in $S^3$ is ribbon. Historically there have been few potential counterexamples due to the difficulty of producing knots that are slice but not obviously ribbon. For recent progress in this direction see \cite{GST} and \cite{AT}. In this paper we study slice disks in $B^4$ whose complements fiber over the circle. The fiber will be a 3--manifold with surface boundary, and the boundary of the slice disk will be a fibered knot in $S^3$.

\renewcommand{\thefootnote}{\fnsymbol{footnote}} 
\footnotetext{\textbf{AMS Subject Classification:} 57M25, 57Q45, 57R65}     
\renewcommand{\thefootnote}{\arabic{footnote}} 

A potentially intermediate class of knots between slice and ribbon are homotopy-ribbon knots, which by definition bound slice disks whose complements admit a handle decomposition without handles of index three or higher. Classical work by Casson-Gordon \cite{casson-gordon:fibered_ribbon} and Cochran \cite{cochran:ribbon_knots} shows that for fibered 1--knots and 2--knots this condition can be characterized in terms of certain properties of the fiber. Furthermore, Casson and Gordon show that a fibered homotopy-ribbon 1--knot $K$ bounds a fibered homotopy-ribbon disk in a homotopy 4--ball. Doubling the disk gives a fibered homotopy-ribbon 2--knot in a homotopy 4--sphere. To understand this relationship better we were motivated to study the intermediary case of fibered homotopy-ribbon disks. Our first result is a characterization of such disks in terms of their fiber.

\begin{theorem}\label{thm:disk_knot}
Let $D$ be a fibered slice disk in $B^4$ with fiber $H$.  Then $D$ is homotopy-ribbon if and only if $H\cong H_g$, the solid genus $g$ handlebody.
\end{theorem}

Next we consider how to produce new fibered disks from old ones, and so prove an analogue of the Stallings twist theorem \cite{stallings:constructions}. The proof continues the broader idea of interpreting certain changes to the monodromy of a mapping torus as surgeries on the total space. The particular surgeries we use for fibered disks double to Gluck twists,  so we see the  phenomenon of an infinite order operation (twisting along a disk) collapse upon doubling to an order two operation (twisting along a sphere). As a result, we can show that many different fibered disks double to the same fibered 2--knot.



\begin{theorem}\label{thm:fibered_disk_knot}
Let $D_0\subset B^4$ be a fibered disk with fiber $H$, and let $E\subset H$ be an essential, properly embedded disk that is unknotted in $B^4$.  Then, changing the monodromy by twisting $m$ times along $E$ gives a new fibered disk $D_m\subset B^4$.  Furthermore, 
\begin{enumerate}
\item the collection $\{D_m\}_{m\in\Z}$ of disks obtained from twisting contains infinitely many pairwise inequivalent elements, and
\item the collection $\{\mathcal DD_m\}_{m\in\Z}$ of 2--knots in $S^4$ obtained by doubling contains at most \emph{two} pairwise inequivalent elements.
\end{enumerate}
Additionally, if the fiber $H$ is a handlebody, then the disk exteriors $\{\overline{B^4\setminus\nu D_m}\}_{m\in\Z}$ are all homotopy equivalent.

\end{theorem}

If we start with a fibered ribbon disk $D_0$ for a ribbon knot $\partial D_0$, we can twist $m$ times along an unknotted disk $E$ in the fiber to obtain $D_m$, which will be a \emph{homotopy}-ribbon disk for $\partial D_m$, by Theorem \ref{thm:disk_knot}. However, it's not obvious that $D_m$ must be ribbon. Therefore, we see that the above procedure could, in principle, be used to produce potential counterexamples to the Slice-Ribbon Conjecture.

In the previous theorem we saw that it was possible for a fibered 2--knot to be obtained as the double of infinitely many different fibered disks (in a fiber-preserving way). In the next theorem we show that this is always the case for fibered homotopy-ribbon 2--knots, with the caveat that we cannot guarantee that the disks lie in $B^4$.

\begin{theorem}\label{thm:halving}
Let $\mathcal S$ be a non-trivial fibered homotopy-ribbon 2--knot in $S^4$. Then $(S^4, \mathcal S)$ can be expressed as the double of infinitely many pairs $(W_m, D_m)$, where $D_m$ is a fibered homotopy-ribbon disk in a contractible manifold $W_m$.  Furthermore, infinitely many of the $W_m$ are pairwise non-diffeomorphic. 
\end{theorem}

Given a 2--knot $\mathcal S\subset S^4$, we call a 1--knot $K$ a \emph{symmetric equator} of $\mathcal S$ if $\mathcal S$ is the double of a disk along $K$ (in some contractible 4--manifold).  We have the following immediate corollary to Theorem \ref{thm:halving}.

\begin{corollary}\label{coro:symmetric_equators}
Any non-trivial, fibered homotopy-ribbon 2--knot has infinitely many distinct fibered symmetric equators.
\end{corollary} 

The techniques in this paper can be illustrated by considering the classical construction of spinning a fibered 1--knot. Spins of fibered knots provide examples to which Theorems \ref{thm:fibered_disk_knot} and \ref{thm:halving} can be applied in a very nice way.  
For example, the collection of 2--knots produced by Theorem \ref{thm:fibered_disk_knot} contains only one isotopy class if $D_0$ is a half-spun disk (see Section \ref{section:spinning}).






Finally, we show that spins of fibered 1--knots are related in an interesting way.

\begin{theorem}\label{thm:torus_surgery}
Suppose $\mathcal S_1$ and $\mathcal S_2$ are 2--knots in $S^4$ obtained by spinning fibered 1--knots of genus $g_1$ and $g_2$, respectively. If $g_1 = g_2$, then $(S^4, \mathcal S_2)$ can be obtained from $(S^4, \mathcal S_1)$ by surgery on a link of tori in the exterior of $\mathcal S_1$. If $g_1 \neq g_2$, we can get from $(S^4, \mathcal S_1)$ to $(S^4, \mathcal S_2)$ by a sequence of \emph{two} sets of surgeries on links of tori.
\end{theorem}

\subsection{Organization}\ 

In Section \ref{section:preliminaries}, we set up basic notation and conventions and introduce the main objects of study.  In Section \ref{section:homotopy}, we describe how the work of \cite{casson-gordon:fibered_ribbon} allows us to pass from fibered homotopy-ribbon 1--knots to 2--disk-knots, and then 2--knots upon doubling.  We prove Theorem \ref{thm:disk_knot} using the characterization of fibered homotopy-ribbon 2--knots in \cite{cochran:ribbon_knots}. 

A main theme of the paper is to interpret changes to the monodromy of fibrations as surgeries on the total space, as has classically been done with the Stallings twist.  In Section \ref{section:mono}, we explore this theme in depth, and develop some lemmata required for our main results.  In Section \ref{section:results}, we prove and discuss Theorems \ref{thm:fibered_disk_knot} and \ref{thm:halving}, and raise a number of interesting questions.   To illustrate the techniques and results found throughout the paper, we conclude with a discussion of spinning fibered 1--knots in Section \ref{section:spinning}, where we prove Theorem \ref{thm:torus_surgery}. 

\subsection{Acknowledgments}\ 

The authors would like to thank Robert Gompf, Cameron Gordon, and Chuck Livingston for helpful comments and conversations.  The authors were supported by NSF grants DMS-1148490, DMS-1309021, and DMS-1400543. Additionally, the authors would like to thank the referees for helpful comments.

\section{Preliminaries and Notation}\label{section:preliminaries}

All manifolds will be assumed to be oriented and smooth, and all maps will be smooth. The boundary of a manifold $X$ will be denoted $\partial X$. If $X$ is closed, we will denote by $X^\circ$ the manifold obtained by puncturing $X$ (that is, removing the interior of a closed ball from $X$). The \emph{double} of $X$ is the manifold $\mathcal{D}X = X \cup_{\partial X}(-X)$, where the gluing is done by the identity map. Note that we also have $\mathcal{D}X \cong \partial (X \times I)$. Similarly, the \emph{spin} of a closed manifold $X$ is $\mathcal S(X) = \partial(X^\circ\times D^2)$. We will denote the closed tubular neighborhood of an embedded submanifold $N$ of $X$ by $\nu N$.

An $n$--\emph{knot} $K$ is an embedded copy of $S^n$ in $S^{n+2}$. We say that $K$ is \emph{unknotted} if it bounds an embedded $D^{n+1}$ in $S^{n+2}$. The \emph{exterior} of $K$ is $E_K = \overline{S^{n+2} \setminus\nu K}$. An $n$--\emph{disk-knot} $D$ is a proper embedding of the pair $(D^n, \partial D^n)$ into $(D^{n+2}, \partial D^{n+2})$ (we will sometimes refer to these as just \emph{disks}). Observe that $\partial D$ is an $(n-1)$--knot in $\partial D^{n+2} = S^{n+1}$. We say $D$ is \emph{unknotted} if there is an isotopy fixing $\partial D$ that takes $D$ to an embedded disk in $\partial D^{n+2}$ (in particular, this implies that $\partial D$ is unknotted as well). Knots occurring as boundaries of disk-knots are called \emph{slice} knots, and the disk the knot bounds is a called a \emph{slice disk}. Embedded knots and disk-knots are considered up to the equivalence of pairwise diffeomorphism. 

Throughout, we will let $\Sigma_g$ denote the genus $g$ surface, $H_g=\natural_gS^1\times D^2$ denote the genus $g$ handlebody, and $M_g$ denote $\#_gS^1\times S^2$.

Let $Y$ be a compact and connected $n$--manifold with (possibly empty) connected boundary $\partial Y$, and let $\phi : Y \rightarrow Y$ be a diffeomorphism. The \emph{mapping torus} $Y \times_\phi S^1$ of $Y$ is the $(n+1)$--manifold formed from $Y \times I $ by identifying $ y \times \{1\}$ with $ \phi(y) \times \{0\}$ for all $y \in Y$. We see that a mapping torus is a  fiber bundle over $S^1$ with fiber $Y$. The map $\phi$ is called the \emph{monodromy} of the mapping torus. If $\partial Y$ is non-empty, then $\partial (Y \times_\phi S^1)$ is a mapping torus with fiber $\partial Y$ and monodromy the restriction $\phi\vert_{\partial Y}$. We are especially interested in the case where the exterior of a knot or disk admits such a fibration,  so we highlight the following definition.

\begin{definition}
We say an $n$--knot $K$ is \emph{fibered} if $E_K$ has the structure of a mapping torus with the additional condition that the boundary of the mapping torus is identified with $\partial (\nu K) = K \times \partial D^2$ such that the fibration map is projection onto the second factor. An $n$--disk-knot $D$ is \emph{fibered} if $\overline{D^n \setminus \nu D}$ has the structure of a mapping torus (again, trivial on $\partial( \nu D)$). In this case, we see that $\partial D$ is a fibered knot.
\end{definition}

\begin{remark}
While a fibered $n$--knot will have a \emph{punctured} $(n+1)$--manifold as a fiber, we can fill in the punctures with a copy of $S^1 \times D^{n+1}$ to get a mapping torus without boundary. This closed mapping torus can be obtained by surgering the $n$--knot rather than removing it. Therefore, in what follows it may be convenient to switch between these two set-ups.
\end{remark}

More generally we will say that a disk or sphere embedded in an \emph{arbitrary} manifold is fibered if its complement admits the above structure. For example, in this paper we will consider fibered knots in homology 3--spheres and fibered disk-knots in contractible 4--manifolds.  The above set-up generalizes easily to these settings.

\begin{example}\label{ex:Stallings}
Let $K$ be a fibered knot in $S^3$, so $E_K$ admits the structure of a mapping torus $\Sigma_g^\circ \times_\varphi S^1$. The boundary of $\Sigma_g^\circ$ is a longitude of $K$,  so performing 0--surgery on $K$ glues in a disk to each longitude, resulting in a closed surface bundle $S^3_0 (K) = \Sigma_g \times_{\widehat{\varphi}} S^1$. The monodromy $\widehat{\varphi}$ is obtained by extending $\varphi$ across the capped off surface $\Sigma_g$. The simplest example is when $K$ is the unknot, in which case the fibers are disks, and 0--surgery results in $S^2 \times S^1$. 

For an arbitrary fibered knot $K\subset S^3$, performing $(1/n)$--surgery on $K$ results in a homology 3--sphere $S^3_{1/n}(K)$. Let $K'$ be the core of the glued-in solid torus, also called the \emph{dual knot} of the surgery. Since $K'$ is the core of the surgery torus, we see that $E_{K'} = E_K$, and so $K'$ is a fibered knot in $S^3_{1/n}(K)$. Note that the boundary compatibility condition is still satisfied because $K'$ and $K$ share the same longitudes in their shared exterior.
\end{example}




We conclude this section by examining a concept that will be central throughout the paper: the relationship between fibered disk-knots and fibered 2--knots.  Let $H$ be a compact 3--manifold with $\partial H\cong \Sigma_g$.   Let $\phi:H\to H$ be a diffeomorphism, and consider the mapping torus $X_0=H\times_\phi S^1$.  We can isotope $\phi$ so that it fixes a small disk $D^2\subset \partial H$.  This gives us an embedded solid torus $D^2\times S^1$ in $\partial X_0$, and a fixed fibering of $\partial D^2\times S^1$ by a preferred longitude $\lambda=\{pt\}\times S^1$.  See Figure \ref{fig:monodromy_ribbon}(a).

Consider the 4--manifold $X=X_0\cup_fh$ obtained by attaching 4--dimensional 2--handle $h$ along $D^2\times S^1$.  We say that $h$ has framing $k$ if the framing is related to the one induced by the product structure on the $D^2\times S^1$ by taking $k$ full right-handed twists (for negative $k$ take left-handed twists).  Observe that the cocore $D$ of $h$ is a fibered disk in $X$, since removing a neighborhood of $D$ is equivalent to removing $h$, and results in the fibered manifold $X_0$.

If we double $D$ we get a 2--knot $\mathcal S=\mathcal DD\subset \mathcal DX$, and notice that $\mathcal S$ is a fibered 2--knot in $\mathcal DX$ with fiber $M=\mathcal DH$.  The monodromy $\Phi:M\to M$ is the \emph{double} of the monodromy $\phi:H\to H$, so $E_\mathcal S=\overline{\mathcal{D}X \setminus \nu \mathcal{S} }\cong M^\circ \times_\Phi S^1$.


\begin{lemma}\label{lemma:framing}
The pair $(\mathcal{D}X, \mathcal{S})$ depends only on the parity of $k$, and the two possible pairs are related by a Gluck twist on $\mathcal S$.
\end{lemma}

\begin{proof}
The pair $(\mathcal{D}X, \mathcal{S})$ is obtained from $E_\mathcal S$ by gluing in $S^2 \times D^2$. Gluing in a copy of $S^2 \times D^2$ to a 4--manifold with $S^2 \times S^1$ boundary amounts to attaching a 2--handle and a 4--handle, since $S^2 \times D^2$ admits a handle decomposition relative to its boundary with one 2--handle and one 4--handle. We can choose the 2--handle to be $h$ from the previous paragraph; in other words, we can isotope the attaching region of the 2--handle to lie in one hemisphere of $S^2\times S^1$. There is a unique way to attach the 4--handle. Gluck \cite{gluck:two-spheres} showed that up to isotopy there is a unique, non-trivial way to glue in $S^2 \times D^2$.  This corresponds to the unique element $\rho$ in the mapping class group of $S^2\times S^1$ that doesn't extend over $S^2\times D^2$. Choosing framing $k$ corresponds to gluing by the map $\rho^k$, and since $\rho^2$ is isotopic to the identity, the result follows. Lastly, we point out that the manifolds coming from even or odd framings are related by a Gluck twist. (See Section \ref{section:mono} for more details.)
\end{proof}

In this paper we are interested in disk-knots with fiber $H_g=\natural_gS^1\times D^2$ and 2--knots with fiber $M_g=(\#_g S^1 \times S^2)^\circ$, which are clearly related by doubling in the way we have just discussed.  

\begin{figure}
\centering
\includegraphics[scale = .12]{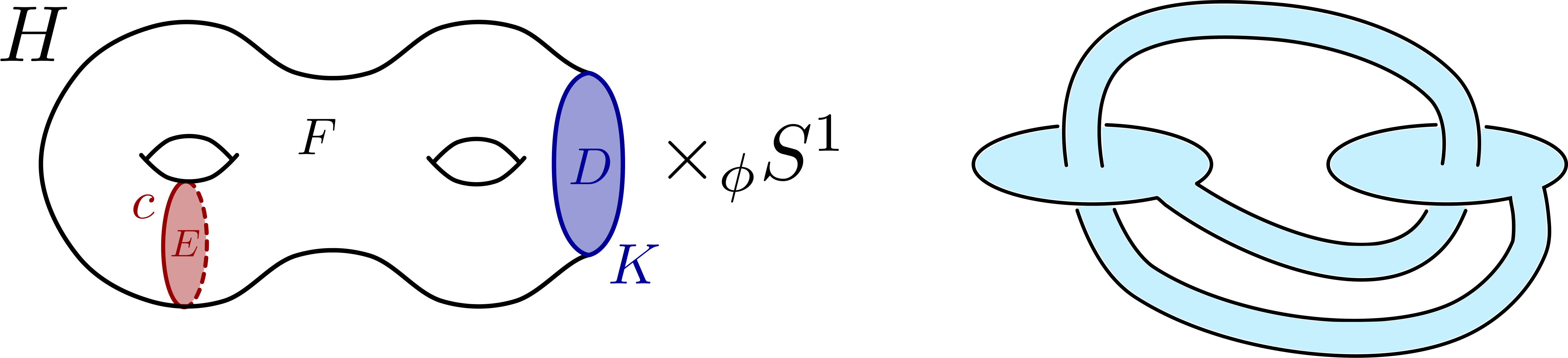}
\put(-335,-25){\large (a)}
\put(-85,-25){\large (b)}
\caption{A schematic (a) for a handlebody bundle, and (b) an immersed ribbon disk in $S^3$.  In (a), $\phi$ has been isotoped to preserve $D$, which we think of as a slice disk for the fibered knot $K$, whose fiber is $F$.  The boundary of the handlebody bundle is the surface bundle $S^3_0(K)$, which has $\widehat F=F\cup_K D$ as fibers.  A Stallings curve $c$ is shown as the boundary of a disk $E$ in $H$.}
\label{fig:monodromy_ribbon}
\end{figure}

Let $D\subset B^4$ be a fibered disk-knot with fiber $H_g$ and monodromy $\phi$, so $\overline{B^4\setminus\nu(D)}\cong H_g\times_\phi S^1$.  Note that $\partial(B^4,D)=(S^3,K)$ is a fibered slice knot with fiber $\Sigma_g^\circ$ and monodromy $\varphi$, where $\varphi=\phi\vert_{(\partial H_g)^\circ}$, as in Example \ref{ex:Stallings} and Figure \ref{fig:monodromy_ribbon}(a). Doubling $(B^4, D)$ results in a fibered 2--knot $\mathcal S \subset S^4$ with fiber $M_g^\circ$. 

The rest of the paper is devoted to the analysis of these objects.

\section{Homotopy-ribbon knots and disks}\label{section:homotopy}

Following \cite{cochran:ribbon_knots}\footnote{Although a more common definition for homotopy-ribbon is what we call \emph{weakly} homotopy-ribbon, we use the definition given by Cochran in \cite{cochran:ribbon_knots} since it makes our statements simpler.}, we will say that an $n$--knot $K\subset S^{n+2}$ is \emph{homotopy-ribbon} if it bounds a slice disk $D \subset D^{n+3}$ with the property that $\overline{D^{n+3} \setminus \nu D}$ admits a handle decomposition with only 0--, 1--, and 2--handles, in which case $D$ is called a \emph{homotopy-ribbon} disk for $K$. The knot $K$ is called \emph{ribbon} if it bounds an immersed disk in $S^{n+2}$ with only ribbon singularities. (See \cite{kawauchi} for details.) These singularities can be removed by pushing the disk into $D^{n+3}$, giving a \emph{ribbon disk} for $K$. See Figure \ref{fig:monodromy_ribbon}(b) for an example.

We call a slice knot $K\subset S^{n+2}$  \emph{weakly homotopy-ribbon} if there is a surjection
$$\pi_1(S^{n+2}\setminus K)\twoheadrightarrow\pi_1(D^{n+3}\setminus D).$$
We have the following inclusions among $n$-knots:
\begin{displaymath}ribbon \subseteq \mbox{\emph{homotopy-ribbon}} \subseteq \mbox{\emph{weakly homotopy-ribbon}} \subseteq slice\end{displaymath}

The Slice-Ribbon Conjecture postulates that every slice knot in $S^3$ is ribbon.  In fact, in the classical case, it is not known whether any of the converse inclusions hold or not.  On the other hand, it is known that every 2--knot is slice \cite{kervaire}, while some 2--knots are not ribbon \cite{yajima}.  We now state the theorems of Cochran and Casson-Gordon that motivate the present work. 

\begin{theorem}[Cochran, \cite{cochran:ribbon_knots}]\label{thm:cochran}
Let $K\subset S^4$ be a fibered 2--knot with fiber $M^\circ$.  Then $K$ is homotopy-ribbon if and only if $M\cong \#_gS^1\times S^2$.
\end{theorem}



\begin{theorem}[Casson-Gordon, \cite{casson-gordon:fibered_ribbon}]\label{CGtheorem}
Let $K\subset S^3$ be a fibered 1--knot with monodromy $\varphi:\Sigma_g^\circ\to\Sigma_g^\circ$. 
\begin{enumerate}
\item If  $K$ bounds a homotopy-ribbon disk $D\subset B^4$, then the closed monodromy $\widehat\varphi:\Sigma_g\to\Sigma_g$ extends over a handlebody to $\phi:H_g\to H_g$.
\item If the closed monodromy $\widehat\varphi$ extends over a handlebody to $\phi:H_g\to H_g$, then there is a homotopy-ribbon-disk $D'$ for $K$ in a homotopy 4--ball $B$ such that $\overline{B\setminus\nu(D')}\cong H_g\times_\phi S^1$.
\end{enumerate}
\end{theorem}


We remark that the original theorem stated by Casson-Gordon is slightly more general, and uses the property we call ``weakly homotopy-ribbon''.  However, the theorem  can be strengthened to conclude homotopy-ribbon, since every mapping torus of $H_g$ can be built with only 0--, 1--, and 2--handles (see Lemma \ref{lemma:handlebody_bundle}). 
\\

We see there is a strong correspondence between fibered homotopy-ribbon knots and conditions on the fiber. Here, we expand the picture to include fibered 2--disk-knots. 

\begin{reptheorem}{thm:disk_knot}
Let $D$ be a fibered disk-knot in $B^4$ with fiber $H$.  Then $D$ is homotopy-ribbon if and only if $H\cong H_g$ for some $g$.
\end{reptheorem}

\begin{proof}
Suppose that $D$ is a fibered homotopy-ribbon disk-knot.  Let $\mathcal S\subset S^4$ be the 2--knot obtained by doubling $D$.  Then $\mathcal S$ is a fibered homotopy-ribbon 2--knot with fiber $M=H\cup_FH$, where $F = \partial H$.  To see this, consider the product $(B^4 \times I, D \times I)$ of $(B^4, D)$. The 3--ball $D \times I$ is a slice disk in $B^5$ for $\mathcal S = \mathcal D D = \partial(D \times I)$. In fact, it is a homotopy-ribbon disk;  the exterior of $D \times I$ is obtained by crossing the exterior of $D$ with $I$, and this preserves the indices of the handle decompositions. By Theorem \ref{thm:cochran}, this means that $M\cong M_h=\#_hS^1\times S^2$ for some $h$.   

Now, suppose that the genus of $F=\partial H$ is $g$.  Since $F$ is a closed, separating surface in $M$, it can be compressed $g$ times (by the Loop Theorem \cite{papa}).  Since $M=H\cup_FH$ is a double, these compressions can be done one by one and symmetrically to both sides.  It follows that $H=H_g\#(\#_{\frac{h-g}{2}}S^1\times S^2)$.  However, we claim that there is a surjection $\pi_1(F)\twoheadrightarrow\pi_1(H)$.  This implies that $h=g$, as desired.

To see where this surjection comes from, consider the infinite cyclic cover $\widetilde W$ of $W=\overline{B^4\setminus\nu(D)}$.  We have that $\widetilde W\cong H\times \R$ and $\partial \widetilde W\cong F\times\R$. Furthermore, we have the following identifications of the commutator subgroups:
$$\pi_1(H)\cong \pi_1(\widetilde W)\cong[\pi_1(W),\pi_1(W)]$$ 
and 
$$\pi_1(F)\cong\pi_1(\partial \widetilde W)\cong[\pi_1(Y),\pi_1(Y)],$$
where $Y=\partial W\cong S^3_0(\partial D)$ is given by 0--Dehn surgery on the knot $\partial D\subset S^3.$
  
By assumption, $\pi_1(S^3\setminus \partial D)\twoheadrightarrow\pi_1(W)$  (here we use the weaker definition of homotopy-ribbon, which follows from the stronger), which implies that $\pi_1(S^3_0(\partial D))\twoheadrightarrow\pi_1(W)$ (since the class of the longitude includes to the trivial element).  This gives a surjection between the corresponding commutator subgroups.  It follows that $\pi_1(F)\twoheadrightarrow\pi_1(H)$, as desired.

Conversely, suppose that $D$ is a fibered disk-knot with handlebody fiber.  Then, by Lemma \ref{lemma:handlebody_bundle} below, the exterior of $D$ in $B^4$ can be built using only 0--, 1--, and 2--handles.  This, by definition, implies that $D$ is  homotopy-ribbon.

\end{proof}

%



The following lemma appears implicitly in \cite{aitchison-rubinstein} (see also \cite{montesinos}).

\begin{lemma}\label{lemma:handlebody_bundle}
A mapping torus with fiber $H_g$ can be built with only 0--, 1--, and 2--handles.
\end{lemma}

The key point is that an $n$--handle in the fiber gives rise to an $n$-handle and an $(n+1)$--handle in the mapping torus.

It is worth noting that the restriction to homotopy-ribbon disk-knots is important, since any fibered slice knot can bound infinitely many different fibered slice disks in the 4--ball, none of which are homotopy-ribbon.  To see this, take any fibered slice disk $D$ in the 4--ball with boundary some fibered slice knot and form the fiber-preserving connected sum $D\#\mathcal{S}$ for any fibered 2--knot $\mathcal{S}$. For example, if we take any fibered 2--knot $\mathcal{S}\subset S^4$ and remove a small 4--ball centered on a  point in $\mathcal{S}$, then the result is a slice disk $D$ for the unknot whose fibers are the same as those of $\mathcal{S}$.


\begin{example}
Many infinite families of handlebody bundles whose total space is the complement of a ribbon disk in $B^4$ were produced by Aitchison-Silver using \emph{construction by isotopy}.  Their main result states that the boundaries of these fibered ribbon disks are doubly slice, fibered ribbon knots that, collectively, realize all possible Alexander polynomials for such knots.  See \cite{ait-silver} for complete details.
\end{example}

The result of Casson-Gordon (Theorem \ref{CGtheorem}) allows us to start with a fibered homotopy-ribbon 1--knot and construct a fibered homotopy-ribbon disk in a homotopy 4--ball $B$. This is accomplished by extending the (closed) monodromy to $H_g$ and adding a 2--handle to the resulting mapping torus. Next, we can double the resulting fibered disk to get a fibered homotopy-ribbon 2--knot in the homotopy 4--sphere $\mathcal{D}B$. 
However, it is not known in general whether $B$ and $\mathcal{D}B$ must be standard. This suggests the following question.

\begin{question}\label{question:inflation_standard}
Does a fibered homotopy-ribbon knot always extend to a fibered homotopy-ribbon disk in $B^4$? Specifically, must $B$ be diffeomorphic to $B^4$?
\end{question}

It is also not clear if the resulting fibered disk and 2--knot depend only on the original 1--knot, or if they also depend on the choice of monodromy extension. We remark that Long has given an example of a surface monodromy that extends over distinct handlebodies \cite{long}.  However, in this case, the corresponding handlebody bundles are diffeomorphic.

  \section{Changing monodromies by surgery}\label{section:mono}

Let $K\subset S^3$ be a fibered knot with fiber $\Sigma_g^\circ$ and monodromy $\varphi$.  An essential curve $c$ on a fiber $F_*$ is called a \emph{Stallings curve} if $c$ bounds a disk $D_c$ in $S^3$ (called a \emph{Stallings disk}), and the framing on $c$ induced by $F_*$ is zero. Given such a curve, we can cut open $E_K$ along $F_*$ and re-glue using the surface diffeomorphism $\tau_c:F_*\to F_*$, which is given by a Dehn twist along $c\subset F_*$. This operation, called a \emph{Stallings twist} along $c$, produces a new fibered knot $K'\subset S^3$ with fiber $\Sigma_g^\circ$ and monodromy $\phi'=\phi \circ \tau_c^{\pm 1}$  \cite{stallings:constructions}.

\begin{figure}
\centering
\includegraphics[scale = .08]{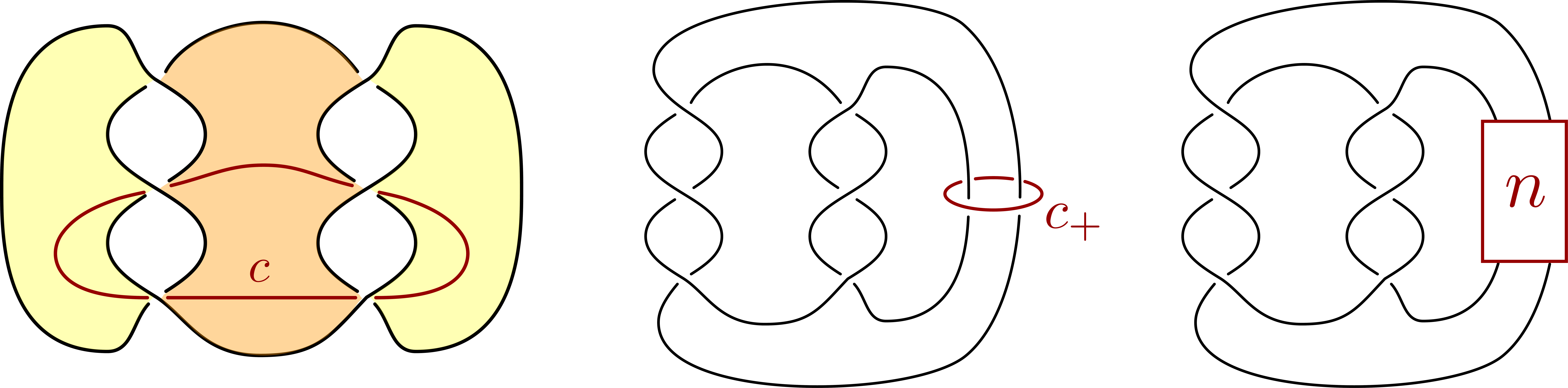}
\put(-335,-25){\large (a)}
\put(-185,-25){\large (b)}
\put(-85,-25){\large (c)}
\caption{An example of Stallings twisting on a fibered ribbon knot along a curve $c$ (shown in (a)) that extends to a disk in the handlebody. In (b), we see a push-off $c_+$ of $c$ (after isotoping the knot), and (c) shows the fibered ribbon knots obtained by twisting $n$ times along $c$. (The box represents $n$ full-twists.)}
\label{fig:Stallings_example}
\end{figure}

The Stallings twist is a classical operation and provides the first instance of a more general theme: interpreting a change to the monodromy of a fiber bundle as a surgery on the total space.  If $K'$ is obtained from $K$ by a Stallings twist, then the resulting mapping torus $E_{K'}$ is related to the original by  $\pm 1$ Dehn surgery on $c$ in $E_K$.  The following lemma allows us to conclude that Stallings twists can be used to create infinite families of distinct knots.

\begin{lemma}\label{lemma:commutator}
Suppose that $K'$ is obtained from $K$ by twisting $m$ times about a Stallings curve $c$ for $K$, and assume that $g(K)\geq 2$.  Then $K'$ is distinct from $K$ provided $|m|=1$ or $|m|>9g-3$.
\end{lemma}

\begin{proof}
If $K'=K$, then the corresponding monodromies are conjugate \cite{burde-zieschang}:  $$\varphi\circ\tau_c^m=\sigma^{-1}\varphi\sigma.$$
This implies that $\tau_c^m$ is a commutator.  Corollary 2.6 of \cite{korkmaz-ozbagci} states that a single Dehn twist is never a commutator, which yields the result when $|m|=1$, and  Corollary 2.4 of \cite{korkmaz} states that $\tau_c^m$ cannot be a commutator when $|m|>9g-3$.
\end{proof}

There are other instances where interpreting a change to the monodromy as a surgery has proved useful.  Gompf \cite{gompf} used this approach to study Cappell-Shaneson 4--spheres, which arise as mapping tori of $T^3$. There,  the corresponding operations are changing the monodromy by twisting along a torus and performing torus surgery (see Section \ref{torus}). Nash \cite{nash} pursued this idea further in his thesis. 

In what follows, we will interpret changing the monodromy of a handlebody bundle by twisting along a disk as a 2--handle surgery, and changing the monodromy of a $\#_g S^1 \times S^2$ bundle by twisting along a sphere as a Gluck twist. We first give some precise definitions of the diffeomorphisms by which we will change the monodromies.

\begin{definition}
 Let $\omega_\theta : D^2 \rightarrow D^2$ and $\Omega_\theta : S^2 \rightarrow S^2$ be the maps given by clockwise rotation of $D^2$ and $S^2$ (about some fixed axis) through the angle $\theta$. Let $H$ be a compact 3--manifold with boundary and let $M$ be a compact 3--manifold.
 
 Given a properly embedded disk $D  \subset H$, we can identify a neighborhood of $D$ with $D^2 \times I$ and define 
 a map $\tau_D : H \rightarrow H$ to be the identity map away from $D^2 \times I$, and $\tau_D (z, t) = (\omega_{2\pi t}(z), t)$ on $D^2 \times I$. 
 
 Similarly, given an embedded 2--sphere $S$ in $M$, we can identify a neighborhood of $S$ with $S^2 \times I$ and
 define a map $\tau_S : M \rightarrow M$ to be the identity map away from $S^2 \times I$, and $\tau_S (x, t) = (\Omega_{2\pi t}(x), t)$
 on $S^2 \times I$. 
 
 We will call these maps \emph{twisting along the disk} $D$ and
 \emph{twisting along the sphere} $S$, respectively. Observe that we can iterate these operations to twist $m$ times along a disk or sphere, for any $m\in\Z$, where $m<0$ should be interpreted as  counterclockwise rotation.
 \end{definition}
 
\begin{remark} \label{rmk:disk_doubling}
The restriction $\tau_D|_{\partial H} : \partial H \rightarrow \partial H$ is a right-handed Dehn twist $\tau_c$ along the curve $c=\partial D^2$. If we double the operation of twisting along a disk we get twisting along a sphere. That is, the double of the map $\tau_D : H \rightarrow H$ is the map $\tau_S : M \rightarrow M$ where $S$ is the double of $D$ in $M=\mathcal{D}H$.
\end{remark}

Later, we will mostly be interested in the case when $H=H_g$ and $M=M_g$, but we will continue to work in generality for the time being.  Next, we introduce two types of surgery and examine how they relate to the diffeomorphisms defined above.

\subsection{The Gluck twist}

 \begin{definition}
 Given an embedded 2--sphere $S$ in a 4--manifold $X$ with trivial normal bundle, we can identify a neighborhood of $S$ with $S^2 \times D^2$.
 We produce a new 4--manifold $X_S$ by removing $S^2 \times D^2$ and re-gluing it using the map $\rho: S^2 \times S^1 \rightarrow S^2 \times S^1$
 defined by $\rho(x, \theta) = (\Omega_{\theta}(x), \theta)$. The manifold $X_S$ is said to be the result of performing a
 \emph{Gluck twist} on $S$.
 \end{definition}

Gluck \cite{gluck:two-spheres} defined the preceding operation and showed that, up to isotopy, $\rho$ is the only non-trivial way to glue in an $S^2 \times D^2$ to a manifold with $S^2\times S^1$ boundary. Furthermore, he showed that $\rho ^2$ is isotopic to the identity map, so $\rho$ has order two in the mapping class group. Therefore, performing two consecutive Gluck twists on a sphere gives you back the original 4--manifold.

Recall that gluing in a copy of $S^2 \times D^2$ to a 4--manifold with $S^2 \times S^1$ boundary amounts to attaching a 2--handle $h$ and a 4--handle. The attaching circle for $h$ will be $\{pt\} \times S^1$, and the framing will be zero (corresponding to the product framing) if we glue by the identity map, or $\pm 1$ if we glue by the Gluck twist map $\rho$. There is a unique way to attach the 4--handle in either case. Furthermore, since $\rho^2$ is trivial, we see that the framing only matters mod 2.

\begin{remark}\label{rmk:gluck}
   Note that $\rho$ is isotopic to a map taking $\Omega_{2\pi t}$ on half of the circle (identifying half the circle with $I$) and the identity map on the other half. Therefore we can substitute such a map in the gluing process without changing the result.
 \end{remark}

Performing a Gluck twist on a 2--knot $\mathcal S \subset S^4$ produces a homotopy 4--sphere $S_\mathcal S$, and it is not known, in general, when $S_\mathcal S$ is diffeomorphic to $S^4$.  For certain types of 2--knots, however, it is known that a Gluck twist not only returns $S^4$, but also preserves the equivalence class of the 2--knot. The easiest case is for the unknotted 2--sphere, which we record here for later use.
  
\begin{lemma}\label{lemma:trivial_gluck}
Let $\mathcal U$ be the unknotted 2--sphere in $S^4$. Performing a Gluck twist on $\mathcal U$ gives back $(S^4, \mathcal U)$. 
\end{lemma}

\begin{proof}
 It is a basic fact that $\overline{S^4 \setminus \nu \mathcal U} = B^3 \times S^1$. If $S_\mathcal U$ is the result of performing a Gluck twist on $\mathcal U$,
 then $S_\mathcal U = B^3 \times S^1 \cup_\rho S^2 \times D^2$. Since $\rho$ clearly extends over $B^3 \times S^1$, we see that $S_\mathcal U=B^3\times S^1\cup_\rho S^2\times D^2$ is diffeomorphic to $B^3 \times S^1 \cup_{id} S^2 \times D^2 \cong S^4$, and the diffeomorphism is the identity map on $S^2\times D^2$ (and so preserves $\mathcal U$). We think of this diffeomorphism as ``untwisting" along $B^3\times S^1$.
\end{proof}

\begin{remark}\label{rmk:trivial_gluck}
In fact, this can be generalized to any 2--knot $\mathcal S$ that bounds $(\#_g S^1 \times S^2)^\circ$. 
See Chapter 13 of  \cite{kawauchi} for details.  It follows, for example, that if $\mathcal S\subset S^4$ is a homotopy-ribbon 2--knot, then $S_\mathcal S\cong S^4$, and the Gluck twist preserves the homotopy-ribbon 2--knot.
\end{remark}

Our main observation here is that changing the monodromy of a 4--dimensional mapping torus by twisting along a sphere contained in a fiber corresponds to performing a Gluck twist on the total space. 

\begin{proposition} \label{prop:sphere_twist}
Let $M$ be a compact 3--manifold, and consider the mapping torus $X = M \times_\Phi S^1$. Let $S$ be an embedded 2--sphere in a fiber $M_*$. Let $X' = M \times_{\phi \circ \tau_S} S^1$ be the mapping torus formed by cutting $X$ along $M_*$ and re-gluing with the diffeomorphism $\tau_S:M_*\to M_*$. Then $X'$ is obtained from $X$ by performing a Gluck twist on $S$. Furthermore, applying $\tau_S$ twice gives $X''=M \times_{\phi \circ \tau_S^2} S^1$, which is diffeomorphic to $X$.
\end{proposition}

\begin{proof}
 Identify a neighborhood of $S$ in $M_*$ with $S \times I$. Since the mapping torus is locally a product $M_* \times I_0$,
 we can identify a neighborhood of $S$ in $X$ with $S \times I \times I_0$. Performing a Gluck twist on $S$ corresponds to deleting
 $S \times I \times I_0$ and re-gluing it by $\rho$ along $S \times \partial(I \times I_0)$. By Remark \ref{rmk:gluck}, we can isotope $\rho$
 so that it is the identity map except on $S \times I \times \{1\}$, where we take $\Omega_{2\pi t}$. The result is clearly diffeomorphic to cutting
 along the fiber $M$
 and re-gluing by $\tau_S$,  so we get the first claim. The second claim follows from the first claim and the fact
 that performing a Gluck twist twice on a sphere preserves the original diffeomorphism type.
\end{proof}

\subsection{2--handle surgery}\label{section:handle}\ 

Next, we introduce an operation on 4--manifolds with boundary that will be instrumental in proving our main results.

\begin{definition}
Suppose a 4--manifold $X$ can be decomposed as $X = X_0 \cup_f h$, with a 2--handle $h$ attached to some 4--manifold with boundary $X_0$ with framing $f$. We say $X'$ is obtained from $X$ by \emph{2--handle surgery}
on $h$ with \emph{slope} $m$ if $X'=X_0\cup_{f'}h$ is formed by removing $h$ and re-attaching it with the same attaching circle but with framing $f'$, where $f'$ is obtained from $f$ by adding $m$ right-hand twists.
\end{definition}

In other words, 2--handle surgery is the process of cutting out and re-gluing a $B^4$ along $S^1 \times D^2$. We make the following simple observations.

\begin{lemma}\label{lemma:invariants}
Suppose $X'$ is obtained from $X$ by 2--handle surgery. Then $\pi_1 (X') \cong \pi_1 (X)$ and $H_*(X)\cong H_*(X')$.
 \end{lemma}

Note that it is possible to change the intersection form of $X$ via 2--handle surgery.

\begin{proof}
 We get a handle decomposition for $X$ by taking a handle decomposition for $X_0$ and adding $h$. A handle decomposition for a 4--manifold  allows one to read off a presentation of the fundamental group: each 1--handle gives a generator, and the attaching circle for each 2--handle  provides a relation. Now the framings for the 2--handles do not affect the relations, so removing $h$ removes a relation and re--attaching $h$ (with any framing) adds the same relation back.  Thus, $\pi_1 (X') \cong \pi_1 (X)$. Similarly, the homology groups can be computed in a simple way from a handle decomposition, and one can check that they don't depend on the framings of the 2--handles.
\end{proof}

The next lemma gives a condition for when a 2--handle surgery preserves the diffeomorphism type of the 4--manifold.  Recall that a properly embedded disk $D$ in a 4--manifold $X$ is said to be \emph{unknotted} if it is isotopic (relative to its boundary) into the boundary of $X$.

\begin{lemma} \label{lemma:unknotted_cocore}
Suppose $X$ contains a 2--handle $h$ whose cocore is unknotted in $X$. Then any 2--handle surgery on $h$ will preserve the diffeomorphism type of $X$.
\end{lemma}

\begin{proof}
Removing $h$ is equivalent to removing a neighborhood of the cocore $D$. If $D$ is unknotted, then this can be thought of as adding a 1--handle to $X$.
We form $X'$ by attaching a 2--handle along the former attaching circle, which will intersect the belt sphere of the 1--handle geometrically once. Therefore, these two handles will cancel (regardless of framing).  It follows that we have not changed $X$, so $X'$ is diffeomorphic to $X$. In other words, $D$ being unknotted is equivalent to the existence of a 1--handle that $h$ cancels, and the cancellation does not depend on the framing of $h$.
 \end{proof}
 
Next, we consider how the boundary of a 4--manifold changes when doing 2--handle surgery. Recall that for $X = X_0 \cup h$, $\partial X$  is obtained from $\partial X_0$ by performing integral surgery on the attaching circle $K$ of $h$, with the surgery coefficient given by the framing.  The belt sphere of $h$ in $\partial X$ is then the dual knot $K'$ of $K$. If $X'$ is the result of the 2--handle surgery, then $\partial X'$ is obtained from $\partial X$ by doing surgery on $K'$, which can be seen as the composition of two surgeries: first do the `dual surgery' from $\partial X$ back to $\partial X_0$, then do another surgery (corresponding to the new framing of $h$) from $\partial X_0$ to $\partial X'$. Note that the surgery from
$\partial X$ to $\partial X'$ won't be integral, in general.
 
\begin{example}\label{ex:Dehn_surgery}
Suppose that $X=B^4$, and $X_0$ is the exterior of a slice disk for a knot $K\subset S^3=\partial B^4$.  Then, $\partial X_0$ is the 3--manifold obtained by 0--surgery on $K$. 
 We see that $X$ is obtained from $X_0$ by attaching a 0--framed 2--handle $h$ along $K'\subset\partial X_0$, where $K'$ is the dual knot to $K$.  Now, $X'$ is formed by $2$--handle surgery on $h$ with slope $m$ if $h$ is attached with framing $m$ instead of framing 0, so $\partial X'$ is $m$--Dehn surgery on $K'$ in $\partial X_0$.  From this viewpoint, we see that $\partial X'$ is obtained by doing $(-1/m)$--Dehn surgery on $K$ in $S^3$. 
\end{example}

Now we show that changing the monodromy of a 4--dimensional mapping torus with boundary by twisting along a disk that is properly embedded in a fiber corresponds to performing a 2--handle surgery on the total space. 

\begin{proposition} \label{prop:disk_twist}
Let $H$ be a compact 3--manifold with boundary, and consider the mapping torus $X = H \times_\phi S^1$.  Let $D$ be a disk that is properly embedded in a fiber $H_*$. Let $X' = H \times_{\Phi \circ \tau_D} S^1$ be the mapping torus obtained by cutting $X$ along $H_*$ and re-gluing with the diffeomorphism $\tau_D:H_*\to H_*$. Then, $X'$ can be obtained from $X$ by performing a 2--handle surgery on a handle $h$ in $X$  where the cocore of $h$ is $D$.
\end{proposition}

\begin{proof}
Similarly to Proposition \ref{prop:sphere_twist}, we identify a neighborhood of $D$ in $H_*$ with $D \times I$ and a neighborhood of $D$ in $X$ as $N = D \times I \times I_0$. Then $N$ is diffeomorphic to $B^4$, and $N \cap (\overline{X \setminus N}) = D \times \partial (I \times I_0)$ is a solid torus. Thus, we can view $N$ as a 2--handle attached to $\overline{X \setminus N}$ with cocore $D \times \{0\} \times \{0\}$. Cutting out $N$ and re-gluing it is then a 2--handle surgery.  We choose the gluing map to be the identity map except on $D \times I \times \{1\}$, where we take $\omega_{2\pi t}$.  This is clearly diffeomorphic to cutting along the fiber and re-gluing by $\tau_D$ to obtain $X'$.
\end{proof}

\section{Main results}\label{section:results}

Now we can apply the techniques from the previous section to prove our main results. Here we give slightly more detailed statements of the theorems described in the introduction.


\begin{reptheorem}{thm:fibered_disk_knot}
Let $D_0\subset B^4$ be a fibered disk-knot with fiber $H$ and monodromy $\phi$, and let $E\subset H$ be an essential, properly embedded disk that is unknotted in $B^4$. Then, the result of twisting $m$ times along $E$ is a new fibered disk-knot $D_m\subset B^4$ with monodromy $\phi_m=\phi\circ\tau_E^m$. Furthermore,
\begin{enumerate}
\item $K_m=\partial D_m$ is a fibered knot in $S^3$ with monodromy $\varphi_m=\varphi_0\circ\tau_c^m$, where $c=\partial E$ is a Stallings curve;
\item the collection $\{D_m\}_{m\in\Z}$ contains infinitely many pairwise inequivalent fibered disk-knots, where $D_0=D$; and
\item the collection $\{\mathcal DD_m\}_{m\in\Z}$ of fibered 2--knots obtained by doubling contains at most two pairwise inequivalent elements.
\end{enumerate}
If in addition the fiber $H$ is the handlebody $H_g$, then the family of disk exteriors $\{\overline{B^4\setminus\nu{D_m}}\}_{m\in\Z}$ will all be homotopy equivalent.
\end{reptheorem}


\begin{proof}
Changing the monodromy by twisting along $E$ is equivalent to performing a 2--handle surgery, by Proposition \ref{prop:disk_twist}, where the cocore of the 2--handle is $E$. Since $E$ is unknotted, Lemma \ref{lemma:unknotted_cocore} states that the diffeomorphism type of the total space doesn't change, and so the result is a fibered disk-knot in $B^4$. In fact there is a small subtlety: we are using the fact that performing the 2--handle surgery in the disk exterior $H_g \times_\phi S^1$ and then filling back in the the disk $D$ is equivalent to performing the 2--handle surgery in $H_g \times_\phi S^ 1\cup \nu D = B^4$. However, the order of these operations is insignificant,  because $E \cap \nu D = \emptyset$.

Because $E$ is unknotted, it can be isotoped to a disk $E'$ lying in $S^3=\partial B^4$ such that $c=\partial E'$ lies on a fiber surface $F_0$ for $K_0=\partial D_0$. Furthermore, the framing on $c$ induced by $F$ is zero, since $c$ bounds a disk in $H_g$. Therefore, $E'$ is a Stallings disk for $K_0=\partial D_0$. Since twisting along a disk restricts to a Dehn twist on the boundary, we see that we are changing the monodromy of the boundary surface bundle by $\tau_E^m\vert_{\partial H}  = \tau_c^m$.  This settles part (1).


Part (2)  follows from the fact that infinitely many of the boundary knots $\partial D_m$ are distinct.   For example, the collection $\{K_{k(9g-2)}\}_{k\in\Z}$ contains pairwise distinct knots, by  Lemma \ref{lemma:commutator}.

Part (3) follows from Remark \ref{rmk:disk_doubling} and Proposition \ref{prop:sphere_twist}.


Finally, if $H=H_g$, then we will show that $Z_m=H_g\times_{\phi\circ\tau_E^m} S^1$ is a $K(G,1)$, where $G\cong\pi_1(Z_m)$ is independent of $m$, by Lemma \ref{lemma:invariants}.  By the long exact sequence of a fibration, and since the base space $S^1$ satisfies $\pi_n(S^1)=0$ for $n>1$, we see that $\pi_n(Z_m)\cong\pi_n(H_g)=0$ for all $n>1$.  Thus, the $Z_m$ are homotopy equivalent for all $m$.
\end{proof}



\begin{remark}

In many cases, it may be that $\mathcal DD_m=\mathcal DD_n$ for all $m,n\in\Z$.  For example, when we consider spinning fibered 1--knots below, we will see examples where the resulting collection of 2--knots contains a \emph{unique} isotopy class.
\end{remark}




In regards to the last statement of the theorem we note that Gordon and Sumners \cite{gordon-sumners} gave examples of disks whose exteriors have the homotopy type of a circle and double to give the unknotted 2--sphere.

Now given a fibered 2--knot $\mathcal S \subset S^4$ with fiber $\#_g S^1 \times S^2$ (i.e. a fibered homotopy-ribbon 2--knot by Cochran), we consider the ways in which it can be decomposed as the \emph{double} of a fibered disk in a contractible 4--manifold. We have already seen in Theorem \ref{thm:fibered_disk_knot} that it is possible for a 2--knot to be the double of infinitely many distinct disk-knots in $B^4$, but this is a somewhat special situation. We next prove a sort of converse result, which is more general but comes with the trade-off that we can no longer guarantee that the fibered disks lie in $B^4$. Again, here we give a more detailed statement than in the introduction.

\begin{reptheorem}{thm:halving}
Let $\mathcal S$ be a non-trivial fibered 2--knot in $S^4$ with fiber $(\#_g S^1 \times S^2)^\circ$. Then $(S^4, \mathcal S)$ can be expressed as the double of infinitely many pairs $(W_m, D_m)$, where 
\begin{enumerate}
\item $W_m$ is a contractible 4--manifold;
\item $D_m$ is a fibered homotopy-ribbon disk-knot in $W_m$;
\item the boundaries $Y_m = \partial W_m$ are all related by Dehn filling on a common 3--manifold; and
\item infinitely many of the $Y_m$ (and, therefore, the corresponding $W_m$) are non-diffeomorphic.
\end{enumerate}
\end{reptheorem}

We remark that the theorem holds in a slightly more general setting.  One could consider $\mathcal S$ to be in a homotopy 4--sphere, for example, yet draw the same conclusions.


To prove Theorem \ref{thm:halving}, we must first record a fact about self-diffeomorphisms of $M_g=\#_g S^1 \times S^2$.

\begin{proposition}\label{prop:mono}
Let $\Phi : M_g \rightarrow M_g$ be an orientation-preserving diffeomorphism. Then $\Phi$ can be isotoped so that it preserves the standard Heegaard splitting of $M_g$. Furthermore, $\Phi$ is isotopic to the double of a diffeomorphism $\phi : H_g \rightarrow H_g$.
\end{proposition}

\begin{proof}
Montesinos \cite{montesinos} gives representatives for generators of the mapping class group $\mathcal M(M_g)$, and, upon inspection, we see that each of the orientation-preserving representatives satisfy the above properties. The result follows, since any orientation-preserving diffeomorphism of $M_g$ is isotopic to a composition of Montesinos' representatives.

\end{proof}

\begin{remark}
Given such a $\Phi$, the description of $\Phi$ as the double of a handlebody diffeomorphism $\phi$ is not unique. Indeed we can alter $\phi$ by twisting twice along any disk and the doubled map $\Phi$ will be unaffected (up to isotopy), by Proposition \ref{prop:sphere_twist}.
\end{remark}

Given a fibered 2--knot $\mathcal{S} \subset S^4$ with fiber $M_g^\circ$ and monodromy $\Phi$, we can apply Proposition \ref{prop:mono} to obtain a (non-unique) handlebody bundle $H_g \times_\phi S^1$ which doubles to the exterior of $\mathcal{S}$, $E_\mathcal{S} = M_g^\circ \times_\Phi S^1$. More precisely, we double the handlebody bundle, except along a disk $D_0$ that is fixed by $\phi$.

Recall from our comments on the Gluck twist that $S^4$ is recovered from $E_\mathcal{S}$ by attaching a 0--framed 2--handle $h$ and a 4--handle. The attaching circle of $h$ is $\{pt\}\times S^1$, and it can be  isotoped to lie on the boundary of the sub-bundle $H_g \times_\phi S^1$ (in fact, we can choose the attaching region to be $D_0\times S^1$). Let $W_0 = (H_g \times_\phi S^1) \cup h$ and observe that $S^4 = \mathcal{D}W_0$. This is most easily seen by noticing that the complement $E_\mathcal S$ is completed to $S^4$ by gluing in a $S^2\times D^2$, where we think of the cocore $D$ of $h$ as the southern hemisphere and the doubled copy of $D$ as the northern hemisphere. Then the proof of Theorem \ref{thm:halving} will be completed by proving the following claim.

\begin{claim}
Let $W_m$ be the manifold obtained by performing 2--handle surgery on $h$ with slope $m$, and let $D_m$ be the cocore of the re-glued $h$. Then, $\mathcal{D}W_m = S^4$, and $\mathcal{D} D_m = \mathcal{S}$. Furthermore, the $W_m$ are contractible 4--manifolds, and infinitely many of the $W_m$ are distinct.
\end{claim}

\begin{proof}
The fact that $\mathcal{D}W_m = S^4$ and $\mathcal{D} D_m = \mathcal{S}$ follows from Lemma \ref{lemma:framing} and Remark \ref{rmk:trivial_gluck}: the pairs $(\mathcal DW_m, \mathcal DD_m)$ are related by Gluck twists on $\mathcal DD_m$, and Gluck twists on 2--knots bounding $M_g^\circ$ preserve both the 4--manifold and the 2--knot. Next, notice that $\pi_1(H_g \times_\phi S^1)\cong\pi_1(M_g^\circ \times_\Phi S^1)$, since $\Phi_*$ and $\phi_*$ are identical as automorphisms of the free group on $g$ generators.  It follows that the HNN extensions presenting these groups will be identical. Since $\pi_1(M_g\times_\Phi S^1\cup h)\cong 1$, it follows that $\pi_1(H_g\times_\phi S^1\cup h)\cong 1$.

One can calculate that $W_m$ has trivial homology, and it then follows from Whitehead's theorem that $W_m$ is contractible. (The inclusion of a point into $W_m$ is a homology equivalence; hence, a homotopy equivalence.)

The boundaries $Y_m=\partial W_m$ are all related by Dehn filling on a common 3--manifold $Y$, as discussed in Example \ref{ex:Dehn_surgery}.  More precisely, we have that $Y_m=Y(-1/m)$.

First suppose that $Y$ is Seifert fibered and that the induced slope of the fibering of the boundary is $a/b$.  If $a\not=0$, then the $(-1/m)$--filling introduces a new exceptional fiber of multiplicity $\Delta(a/b,-1/m)=am-b$, and the spaces are thusly distinguished.  If $a=0$, then $0$--filling is reducible \cite{seifert}; however, this manifold is also a nontrivial closed surface bundle, a contradiction.

If $Y$ is hyperbolic, then we have $vol(Y(-1/m))<vol(Y)$ for all $m$ and, by \cite{thurston}, $$\lim_{m\to\infty}vol(Y(-1/m))=vol(Y).$$  It follows that there is an infinite sequence $\{m_i\}_{i\in\N}$ such that $vol(Y_{m_1})<vol(Y_{m_2})<vol(Y_{m_3})<\cdots$.  Therefore, infinitely many of the $Y_m$ are distinct.  If $Y$ is toroidal, then we can cut $Y$ along all essential tori, and repeat the above argument on the atoroidal piece containing $\partial Y$.
\end{proof}
 
\begin{question}
Is it possible that $Y_{m'}\cong Y_{m}$ for some $m'\not=m$?  This would be an example of a cosmetic surgery on a fibered homotopy-ribbon knot in a $\Z$--homology 3--sphere.
\end{question}



Quite a few choices were made in the preceding construction, and this raises a number of questions. An obvious one is whether we can always arrange for one of the $W_m$ to be the standard 4--ball.  We can phrase the question as follows.

\begin{question}
Is every fibered homotopy-ribbon 2--knot in $S^4$ the double of a fibered homotopy-ribbon disk-knot in $B^4$?
\end{question}

A result of Levine \cite{levine:disk_knots} states that the double of any disk-knot is doubly slice, so to answer the above question negatively, it would suffice to find a fibered homotopy-ribbon 2--knot in $S^4$ that is not doubly slice. If we generalize to arbitrary fibers then we can observe that not all fibered 2--knots are doubles of fibered disks. This follows simply because there exist 2--knots that are fibered by 3--manifolds that are themselves not doubles (see \cite{zeeman}).

We introduce the term \emph{halving} to describe the process of going from a fibered homotopy-ribbon 2--knot $\mathcal S\subset S^4$ to one of the cross-sectional fibered homotopy-ribbon 1--knots produced by Theorem \ref{thm:halving}.  This raises the question of which 1--knots can result from this process for a given 2--knot.

\begin{definition}
Given a 2--knot $\mathcal S\subset S^4$, we call 1--knot $K$ a \emph{symmetric equator} of $\mathcal S$ if $K$ can be obtained from $\mathcal S$ by halving.  In other words, $\mathcal S$ is the double along $K$ of a disk in a contractible 4--manifold.
\end{definition}

From this point of view, we have the following corollary to Theorem \ref{thm:halving}.

\begin{repcorollary}{coro:symmetric_equators}
Any nontrivial, fibered homotopy-ribbon 2--knot has infinitely many distinct fibered symmetric equators.
\end{repcorollary}


\section{Spinning fibered 1-knots}\label{section:spinning}

Here we give a method to produce many examples of fibered ribbon disks and fibered ribbon 2--knots. Recall that the \emph{spin} of a manifold $X$ is given by $\mathcal S(X)=\partial(X^\circ\times D^2)$.  Equivalently, we can view the spin as a double, $\mathcal S(K)\cong \mathcal D(X^\circ\times I)$, as follows.
$$\mathcal D(X^\circ\times I) = \partial(X^\circ\times I\times I)\cong\partial(X^\circ\times D^2)$$

Notice that $\mathcal S(S^n)\cong S^{n+1}$, so the the spin of a knot $(S^3, K)$ is a 2--knot $(S^4, \mathcal S(K))$, called the \emph{spin} of $K$.  Also, if we let $(B^3, K^\circ)=(S^3,K)^\circ$ be the punctured pair, then $(B^3,K^\circ)\times I=(B^4, D_K)$, where $D_K$ is a ribbon disk for $K\#(-K)$, which we sometimes call the \emph{half-spin} of $K$.  The justification of this terminology is that $\mathcal D(B^4, D_K)\cong (S^4, \mathcal S(K))$.

Now, if $(S^3, K)$ is a fibered knot, then $K\#(-K)$ is a fibered ribbon knot, $(B^4, D_K)$ is a fibered ribbon disk for $K\#(-K)$, and $(S^4, \mathcal S(K))$ is a fibered ribbon 2--knot.  Suppose that $K$ has monodromy $\varphi:\Sigma_g^\circ\to\Sigma_g^\circ$.  Then, $(B^4, D_K)\cong (S^3,K)^\circ\times I$ is clearly fibered with fibers $H_{2g}\cong \Sigma_g^\circ\times I$ and monodromy $\phi=\varphi\times\id$.  It is clear that the fibration restricts to the boundary to give a fibration of $(S^3, K\#(-K))$, where the fibers are $\Sigma_g^\circ\natural(-\Sigma_g^\circ)$ and the monodromy is $\varphi\natural(-\varphi)$.  Finally, if we view $(S^4, \mathcal S(K))$ as the double of $(B^4, D_K)$, then we see that the former is fibered with fiber $M_{2g}=\#_{2g}S^1\times S^2$ (since $M_{2g}=\mathcal DH_{2g}$) and monodromy $\Phi=\mathcal D\phi$. This spinning construction provides a nice set of examples  to apply the techniques and results from the previous sections.

The trivial pair $(S^4,S^2)$ admits a natural fibration by 3--balls.  One way to visualize the spin $\mathcal S(K)$ is to view $K$ as a knotted arc $K^\circ\subset B^3$, and to identify this 3--ball with a fiber of $(S^4, S^2)$.  Then, $\mathcal S(K)$ is the trace of $K^\circ$ as the 3--ball sweeps out all of $S^4$ via this fibration.  See Figure \ref{fig:spun_arcs}(a).

\begin{figure}
\centering
\includegraphics[scale = .15]{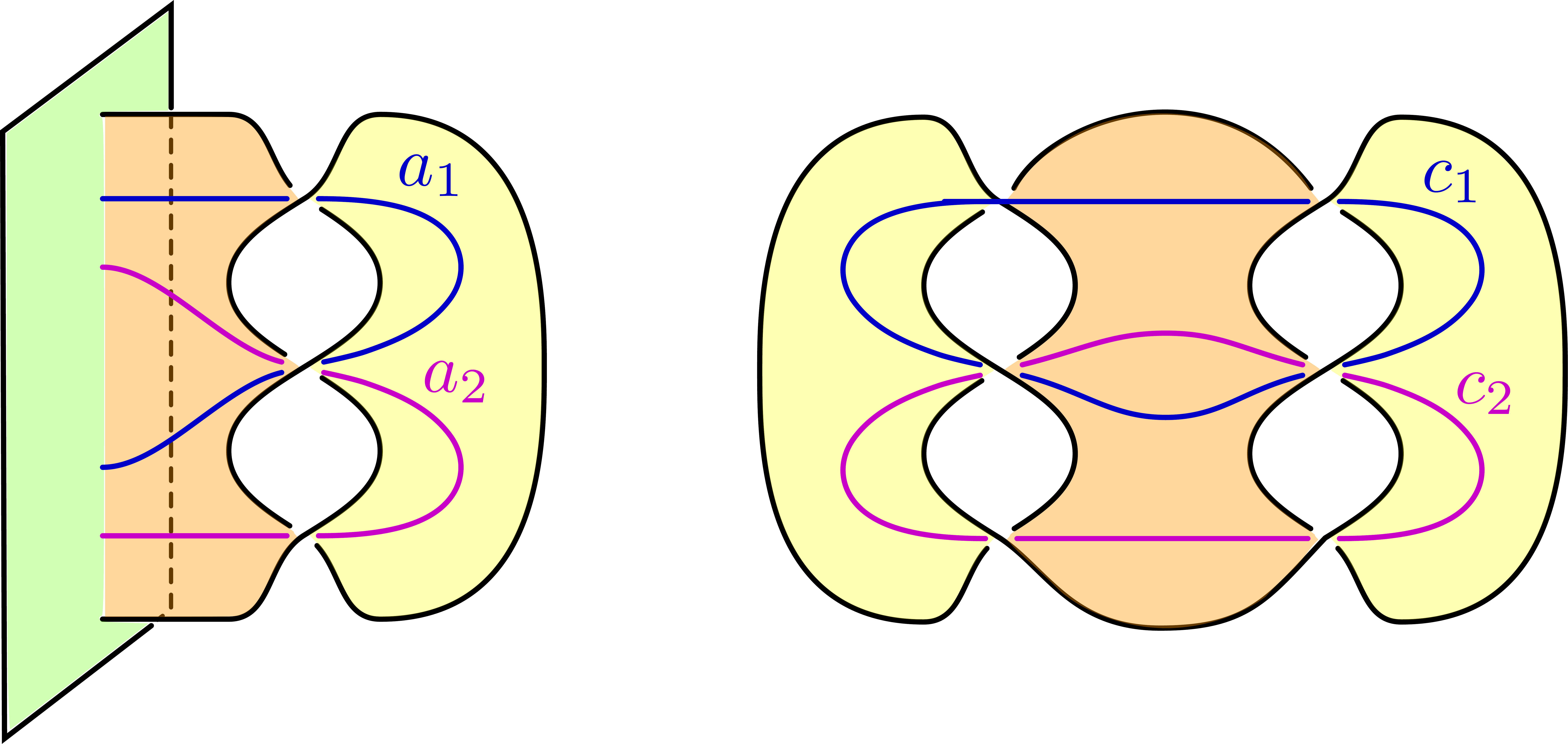}
\put(-298,0){\large(a)}
\put(-102,0){\large(b)}
\caption{The right handed trefoil, shown (a) as an embedded arc in $B^3$. (A portion of the boundary 2--sphere is shown in green.)  Lying on a Seifert surface for $K^\circ$, we see two arcs, $a_1$ and $a_2$, which are boundary parallel in $B^3$.  If we form the spin of the the picture in (a), we get the spun trefoil knot, together with two spheres $S_1$ and $S_2$ corresponding to the spins of the arcs $a_1$ and $a_2$.  In (b), we see an equatorial cross section of the spun trefoil.  Alternatively, (b) shows the boundary of the half-spun trefoil.  The curves $c_1$ and $c_2$ are Stallings curves for the square knot, and bound disks in the genus two handlebody, thought of as a fiber of the half-spun trefoil.}
\label{fig:spun_arcs}
\end{figure}

Let $K$ be a fibered knot, and let $F$ be a fiber of $K^\circ\subset B^3$.  Let $a$ be an unknotted arc on $F$, with endpoints in $\partial B^3$, as in Figure \ref{fig:spun_arcs}(a).
Then $D_a = a \times I \subset (\overline{B^3 \setminus\nu K}) \times I$ is an unknotted disk in a fiber $H_{2g}$ of $D_K$. We can apply Theorem \ref{thm:fibered_disk_knot} to obtain new fibered disks $D_{J_m}$ in $B^4$ by twisting $m$ times along $D_a$. The disks $D_{J_m}$ will be homotopy-ribbon disks for knots $J_m$ obtained from $K \# (-K)$ by doing $m$ Stallings twists along $c=\partial D_a$ (see also Figure \ref{fig:Stallings_example}).

Choosing different arcs for $a$ provides a wealth of examples.  Some of these families of fibered ribbon knots have been studied elsewhere \cite{ait,ait-silver,hitt-silver,stallings:constructions}, often with the goal of finding infinitely many distinct fibered ribbon knots with, say, the same Alexander module.

When we double this picture to get $(S^4, \mathcal S(K))$, we get an unknotted sphere $S_a = \mathcal{D}D_a$ inside a fiber $M\cong M_{2g}$ of $\mathcal S (K)$. We see that $S_a$ is the spin of $a$, and twisting along $S_a$ gives a new fibered ribbon 2--knot $\mathcal S(K)'$ in $S^4$, which is the double of the disk-knot $D_{J_m}$, for all odd $m$ (Theorem \ref{thm:fibered_disk_knot}). However, in this special setting, we see that twisting on $S_a$ preserves the 2--knot $\mathcal S(K)$.

\begin{proposition}
$\mathcal S(K)' = \mathcal S(K)$
\end{proposition}

\begin{proof}
Since $a$ is unknotted it bounds a semi-disk in $B^3$ (that is, $a$ together with an arc on the boundary bound a disk). Perturb the disk so that it intersects $K^\circ$ transversely in $k$ points. Spinning this semi-disk gives a ball $B_a$ bounded by $S_a$, showing that $S_a$ is, in fact, unknotted in $S^4$. Furthermore, we see that $B_a$ intersects $\mathcal S(K)$ in $k$ circles, which come from spinning the $k$ intersection points. The circles form a $k$ component unlink since they are the spins of isolated points. By an isotopy of $\mathcal S(K)$ taking place in a collar $B_a \times I$, we can assume the circles all lie concentrically in a plane in $B_a$. 

Now, $\mathcal S(K)'$ is formed by changing the monodromy by twisting along $S_a$, and, by Proposition \ref{prop:sphere_twist}, this is equivalent to performing a Gluck twist on $S_a$. By the proof of Lemma \ref{lemma:trivial_gluck}, we see that the diffeomorphism taking $S^4_{S_a}$ (the result of performing a Gluck twist on $S_a$) back to $S^4$ is ``untwisting" along the complement of $S_a$. This can be thought of as cutting $S^1 \times B^3 = \overline{S^4 \setminus \nu S_a}$ at $B_a$ and re-gluing by a $2\pi$ twist about an axis, which we can choose to be perpendicular to the plane in $B_a$ containing the concentric circles. Therefore the twist in fact preserves the circles, and hence sends $\mathcal S(K)'$ onto $\mathcal S(K)$, proving the claim.
\end{proof}

\subsection{Torus surgery}\label{torus}

We finish by discussing one more type of monodromy change. Just as one can change the monodromy of a fibered 2--knot by twisting along a sphere, one can also change the monodromy by twisting along a torus. This can be interpreted as a torus surgery in the total space, and there are some nice applications of this for spun knots. 

\begin{definition}
Given a torus $T$ embedded in compact 3--manifold $M$, we can identify a neighborhood of $T$ with $S^1 \times S^1 \times I$. We define the map $\tau_T :M \to M$ to be the identity map away from a neighborhood of $T$, and the identity map on the first $S^1$ factor times a Dehn twist on $S^1 \times I$ in the neighborhood $S^1 \times S^1 \times I$. We will call this map \emph{twisting along the torus} $T$.
\end{definition}

Given an embedded torus $T$ in a 4-manifold $X$, \emph{torus surgery} on $T$ (also called a \emph{log transform}) is the process of cutting out $\nu T$ and re-gluing it by a self-diffeomorphism of $\partial \nu T = T^3$. Torus surgery is a well-studied operation on 4-manifolds, and we refer the reader to \cite{gompf-stipsicz} for a more comprehensive overview. Here we state its relationship to changing the monodromy by twisting along a torus (see \cite{gompf} and \cite{nash}) . The proof is similar to Proposition \ref{prop:sphere_twist} and Proposition \ref{prop:disk_twist}.

\begin{proposition} \label{prop:torus_twist}
Let $Y$ be a compact 3--manifold, and consider the mapping torus $X = Y \times_\phi S^1$. Let $T$ be an embedded torus in a fiber $M_*$. Let $X' = Y\times_{\phi \circ \tau_T} S^1$ be the mapping torus formed by cutting $X$ open along $M_*$ and re-gluing using the diffeomorphism $\tau_T:M_*\to M_*$. Then $X'$ can be obtained from $X$ by performing a torus surgery on $T$.
\end{proposition}

We can show that all fibered 2--knots coming from spinning fibered 1-knots are related by torus surgeries. In this special setting, this can be thought of as an analogue to a theorem of Harer \cite{harer}, which states that all fibered links in $S^3$ are related to the unknot by Stallings twists and a type of plumbing.  

\begin{reptheorem}{thm:torus_surgery}

Suppose $\mathcal S_1$ and $\mathcal S_2$ are 2--knots in $S^4$ obtained by spinning fibered 1--knots of genus $g_1$ and $g_2$ respectively. If $g_1 = g_2$, then $(S^4, \mathcal S_2)$ can be obtained from $(S^4, \mathcal S_1)$ by surgery on a link of tori in the exterior of $\mathcal S_1$. If $g_1 \neq g_2$, we can get from $(S^4, \mathcal S_1)$ to $(S^4, \mathcal S_2)$ by a sequence of \emph{two} sets of surgeries on links of tori.
\end{reptheorem}

\begin{proof}
Assume $\mathcal S_1$ and $\mathcal S_2$ are the result of spinning fibered 1-knots $K_1$ and $K_2$, respectively, each with fiber genus $g = g_1 = g_2$. Let us first restrict our attention to $\mathcal S_1$ and $K_1$. The monodromy for $K_1$ is given by a composition of Dehn twists along some collections of curves $\{\alpha_i\}$ on its fiber $\Sigma_g^\circ$, and so the monodromy for the fibered disk $D_{K_1}$ is a composition of Dehn twists times the identity on its fiber $\Sigma_g^\circ \times I = H_{2g}$. Therefore, we see that the monodromy for $D_{K_1}$ is a composition of maps that are twisting along annuli $\{\alpha_i \times I\}$, and when we double to obtain $\mathcal S_1$, we see that the monodromy of the fibered ribbon 2--knot is a composition of maps that are twisting along tori $\{\alpha_i \times S^1\} \subset E_{\mathcal S_1}$. The same is true, of course, for $\mathcal S_2$, and we can achieve that the collections of tori are disjoint by selecting them to lie in disjoint fibers. Then, $(S^4, \mathcal S_2)$ can be obtained from $(S^4, \mathcal S_1)$ by cutting along these fibers and changing the monodromy by twisting along these collections of tori. By Proposition \ref{prop:torus_twist}, this can be realized by torus surgery on the tori, proving the first part of the theorem.

Now suppose $g_2 > g_1$. Then $\mathcal S_1$ is fibered by $M_{2g_1}$ and $\mathcal S_2$ is fibered by $M_{2g_2}$. The first task is to increase the number of $S^1 \times S^2$ factors in the fiber of $\mathcal S_1$ using torus surgery.  Isotope the monodromy of $\mathcal S_1$ so that it fixes a 3--ball $B$ in $M_{2g_1}$. Then, we have $B \times S^1$ inside the exterior $E_{ \mathcal S_1}$. If $U \subset B$ is the $2(g_2 - g_1)$ component unlink, then $U \times S^1 \subset B \times S^1$ is an unlink of $2(g_2 - g_1)$ tori in $E_{\mathcal S_1}$.  We perform surgery on these tori such that the gluing maps correspond to the gluing map for 0--framed Dehn surgery on $U$ times the identity map in the remaining $S^1$ direction. This changes $B \times S^1$ to $ M_{2(g_2-g_1)}^\circ \times S^1$. The result of this set of torus surgeries is a mapping torus with fiber $M_{2g_2}$. We can now change the monodromy to that of $\mathcal S_2$ by another set of torus surgeries as in the first part of the proof. Hence $(S^4, \mathcal S_2)$ can be obtained from $(S^4, \mathcal S_1)$ by a sequence of two sets of surgeries on links of tori.
\end{proof}

Suppose $c$ is a Stallings curve in a fiber $F_*$ of $K^\circ$ in $\overline{B^3 \setminus \nu K^\circ}$. Then spinning $c$ gives a torus $T_c$ inside a fiber $M_*$ of $\mathcal S(K)$. We have the following relationship.

\begin{corollary}
If $K_c$ is the result of applying a Stallings twist to $K$ along $c$, then $\mathcal S(K_c)$ can be obtained from $\mathcal S(K)$ by torus surgery on $T_c$.
\end{corollary}


\bibliographystyle{amsalpha}
\bibliography{DiskKnotBiblio_Journal.bib}

\end{document}